\documentclass[12pt,a4paper]{amsart}
\usepackage{graphicx,amssymb}
\input xy
\xyoption{all}
\usepackage[all]{xy}
\usepackage{hyperref}

\newtheorem{thm}{Theorem}[section]
\newtheorem{cor}[thm]{Corollary}
\newtheorem{lem}[thm]{Lemma}
\newtheorem{prop}[thm]{Proposition}
\theoremstyle{definition}

\newtheorem{rem}[thm]{Remark}
\newtheorem{exam}[thm]{Example}
\numberwithin{equation}{section}

\newcommand{\QQ}{\mathbb Q}

\newcommand{\ZZ}{\mathbb Z}
\newcommand{\CC}{\mathbb C}
\newcommand{\PP}{\mathbb P}
\newcommand{\FF}{\mathbb F}

\newcommand{\lra}{\longrightarrow}
\newcommand{\hra}{\hookrightarrow}
\newcommand{\ra}{\rightarrow}

\newcommand{\cA}{\mathcal{A}}

\newcommand{\tC}{\widetilde{C}}

\newcommand{\cC}{\mathcal{C}}

\newcommand{\cM}{\mathcal{M}}

\newcommand{\cV}{\mathcal{V}}
\newcommand{\cO}{\mathcal{O}}

\newcommand{\cR}{\mathcal{R}}
\newcommand{\cS}{\mathcal{S}}

\newcommand{\cB}{\mathcal{B}}

 \DeclareMathOperator{\Ker}{Ker}
  \DeclareMathOperator{\Fix}{Fix}

 \DeclareMathOperator{\Nm}{{Nm}}
  
 \DeclareMathOperator{\pr}{{Pr_{2,7}}}

 \DeclareMathOperator{\im}{Im}
 
\DeclareMathOperator{\End}{{End}}

\begin{document}

\title[The fibers of the Prym map]{The fibres of the Prym map of \\ \'etale cyclic coverings of degree 7 }
\author{ Herbert Lange and Angela Ortega}
\address{H. Lange \\ Department Mathematik der Universit\"at Erlangen \\ Germany}
\email{lange@mi.uni-erlangen.de}
              
\address{A. Ortega \\ Institut f\" ur Mathematik, Humboldt Universit\"at zu Berlin \\ Germany}
\email{ortega@math.hu-berlin.de}

\thanks{The second author is partially supported by the {\it Sonderforschungsbereich} 647 ``Raum-Zeit-Materie"}

\begin{abstract} 

We study the Prym varieties arising from \'etale cyclic coverings of 
degree 7 over a curve of genus 2. We prove that these Prym varieties are
products of Jacobians $JY \times JY$ of genus 3 curves Y  with  polarization
type $D=(1,1,1,1,1,7)$. We describe the fibers of the Prym map between the moduli space of 
such coverings and the moduli space of abelian sixfolds with
polarization type $D$, admitting an automorphism of order 7.
\end{abstract}

\maketitle

\section{Introduction}

Given a finite covering $f:\tC \ra C$ between smooth projective curves, one can associate to $f$ an abelian subvariety of the Jacobian $J\tC$ by taking
the connected component containing the zero of the kernel of the norm map $\Nm_f : J\tC \ra JC$,  $[D] \mapsto [f_*(D)]$. The resulting variety 
$$
P(f):= (\Ker \Nm_f )^0 \subset J\tC
$$  
is called the {\it Prym variety of $f$}. The restriction of the theta divisor of $J\tC$ to $P(f)$ defines a polarization on the Prym variety which is 
known to be twice a principal  polarization
when $f$ is an \'etale double covering (see \cite{mu}). The assignment $[f: \tC \ra C] \mapsto P(f)$ yields a  {\it Prym map} between the 
moduli space of the corresponding coverings and the moduli space of abelian varieties of suitable dimension and
polarization (not necessarily principally polarized). In very few cases the Prym maps 
are generically finite over its image (see \cite{b}, \cite{ds},  \cite{f}, \cite{lo} and \cite{lo2} ) and often the structure of the fibers 
can be understood in  geometrical terms (\cite{d}, \cite{f}).  

We study the Prym varieties associated to a non-trivial  $7$ cyclic \'etale covering $f: \tC \ra C$ over a curve $C$ of genus 2. 
The corresponding Prym variety
$P(f)$ is a 6-dimensional abelian variety with polarization type $D=(1,1,1,1,1,7)$.  Let $\cR_{2,7}$ denote the 
moduli space of these coverings and $\cB_D$ the moduli space of abelian varieties of dimension $6$ with a polarization of  type $D$ and an 
automorphism of order $7$ compatible with the polarization. In \cite{lo2} we proved that the Prym map  $\pr: \cR_{2,7} \ra \cB_D$ is dominant 
and since both moduli spaces are 3-dimensional (see \cite{lo2} for a proof), the map $\pr$ is generically finite and in fact, of degree 10. 
In this article we give a geometric description of the fibers of the Prym map $\pr$. 

First we prove in Section 2 that for all $[f: \tC \ra C] \in \cR_{2,7}$ the associated Prym variety $P(f) = JY \times JY$, 
where $Y$ is a curve of genus 3 whose Jacobian admits a multiplication in the totally real cubic subfield $\QQ(\zeta_7)^0$ of the cyclotomic
field $\QQ(\zeta_7)$ (Proposition \ref{realmult}); moreover $Y$ is uniquely determined by the Prym variety. In Section 3 we show that the 
curve $Y$ admits a degree 7-covering $Y \ra \PP^1$ with ramification type $(2,2,2,1)$, that is, $f$ is simply ramified  over each branch point
at 3  points and unramified at 1.  Let  $\cC_3$ denote  the locus in $\cM_3$ of curves $Y$ having a covering of this type.  
The main result is the following.

 \begin{thm} \label{main-thm}
The elements of a generic fiber of the Prym map $\pr:\cR_{2,7} \ra \cB_D$ are in bijection with the set of degree 7 coverings 
$\overline f: Y \ra \PP^1$ with ramification type $(2,2,2,1)$ over 6 general points that a curve $Y \in \cC_3$ admits. 
\end{thm}

\begin{rem}
In \cite{lo2} we considered a partial compactification of $\pr$ which is proper and generically finite of degree 10. So, a finite fibre consists of 10 elements counted 
with multiplicities. This fact together with Theorem \ref{main-thm} 
shows that the number of degree 7 coverings $\overline f: Y \ra \PP^1$ with ramification type $(2,2,2,1)$ 
over each of the 6 ramification points is $\leq$ 10.
\end{rem}

 \bigskip
 
{\it Acknowledgements}. We would like to thank Alice Silverberg and Karl Rubin for their
help with the proof of Lemma 2.7.

\section{Description of the Prym variety as product of Jacobians}

Let $C$ be a genus 2 curve and $f:\tC \ra C$ a non-trivial \'etale cyclic covering of degree 7.
We shall give an explicit description of the associated Prym variety $P(f)$. Part of the results in this section are contained in \cite{o}.  
It is known that  the hyperelliptic involution $\iota$ of $C$ lifts, though non-canonically, to an involution $j:\tC \ra \tC$.  If we denote by $\sigma$ 
the automorphism of order 7 on the cover $\tC$,
then $j,\sigma$ generate the dihedral group 
$$
D_7= \langle j, \sigma \ \mid \ j^2=\sigma^7 =1, j\sigma j= \sigma^{-1}  \rangle.
$$
Moreover, the composed map $\tC \ra \PP^1$ is Galois with group $D_7$ (\cite[Proposition 2.1]{bl2}). Using the commutativity of the diagram
\begin{equation} \label{diag1}
\xymatrix{
\tC \ar[d]_f \ar[r]^{j} &  \tC \ar[d]^f  \\
C \ar[r]^{\iota} &  C
}
\end{equation}
one checks that the images of the fixed points of $j$ are precisely the fixed points of  $\iota$ and, since $f$ is \'etale, there is only one fixed 
point  on each fiber of a fixed point of $\iota$. Let 
$$
g: \tC \ra Y:= \tC/ \langle j \rangle
$$ 
be the double covering  associated to the involution $j$. Then $g$ is 
ramified exactly at the 6 fixed points of $j$, so by Hurwitz formula $Y$ is of genus 3.  Notice that since $g$ is a ramified cover, the map 
$g^*: JY \ra J\tC$ is injective, so we can regard  $JY$ as subvariety in $J\tC$.  As before, let $P=P(f)$ be the Prym variety associated to $f$. 

\begin{prop} \label{2.1}
The map 
$$
\begin{array}{ccc}
\psi: JY  \times JY& \ra & P \\
\qquad (x,y) &  \mapsto &  x+\sigma(y) 
\end{array}
$$
is an isomorphism  of abelian varieties.
\end{prop}

Certainly $\psi$ does not respect the canonical polarizations, since $JY \times JY$ is canonically principally polarized whereas the induced 
polarization $\Xi$ on $P$ is of type $(1,\dots,1,7)$.

\begin{proof}
First we check that $JY$ and  $\sigma(JY)$ are contained in $ P$.  Recall that $P = \Ker(1 + \sigma +\cdots + \sigma^6)^0$ and 
$g^*(JY) = \im (1+j)$.  On $J\tC$
$$
(1 + \cdots + \sigma^6)(1+j)= 1 + \cdots + \sigma^6+ j + \cdots + j\sigma^6=0
$$
since this map is the the norm map of $\tC \ra \PP^1$. This shows that $JY \subset P$ and as $\sigma$ acts on $P$, we also 
have $\sigma(JY) \subset P$. 

Note that $\dim (JY \times JY) = \dim P$, so in order to prove the proposition it suffices to show that $\psi$
is injective. Recall that, since $Y=\tC / \langle j \rangle$, the elements of  $JY \simeq g^*JY$ are fixed by $j$.  
Let $(x,y) \in JY \times JY$ such that $x+\sigma (y) =0 $. Then 
$$
x= j(x) = -j\sigma (y) = -\sigma^6 j(y)= -\sigma^6(y)
$$
and hence $\sigma^2(x) = - \sigma (y) = x$. This shows that $x \in \Fix (j, \sigma)$. In particular $x\in \Fix(\sigma) \cap P \subset J\tC[7]$.
Since $\sigma (x) =x$ and $f$ is \'etale there exists an element $z \in JC$
such that $x=f^*(z)$. Further, $j(x)=x$ and the commutativity of diagram \eqref{diag1} implies that 
$$
x= j(x)= jf^*(z) =f^*\iota (z) = -f^* (z) = -x
$$ 
since on $JC$ we have $\iota (z) = -z $, that is $x \in J\tC[2]$. In conclusion $x \in J\tC[2] \cap J\tC[7] = \{0\}$ and therefore $\psi$ is injective.
\end{proof}

\bigskip 

In order to describe the polarization on the product let  $\Xi$ denote the polarization of type $(1,1,1,1,1,7)$ on $P$. Hence
its pull back $\psi^*\Xi$ is also  of type $(1,1,1,1,1,7)$. Now $\Xi$ is the restriction of the canonical 
polarization of $J\tC$ to $P$. So we may consider the polarization $\psi^*\Xi$ as the pullback 
of the canonical polarization of $J\tC$ to $JY \times JY$.

We identify $\widehat {J\tC} = J\tC$ and $\widehat {JY} = JY$ via the canonical 
polarizations. The induced split polarization on $JY \times JY$ gives an identification 
$(JY \times JY)^{\wedge} = JY \times JY$.
Let $\theta: JY \times JY \stackrel{\psi}{\ra} P \hra J \tC$ be the embedding. For $i = 1, 2$ denote by $\theta_i$ the composition
$$
\theta_i: JY \ra JY \times JY \stackrel{\theta}{\ra} J\tC,
$$
where the first map is the natural embedding into the $i$-th factor.  One checks  that 
$$
\theta_1 = g^*: JY \ra J\tC \quad \mbox{and} \quad \theta_2 = \sigma \circ g^*: JY \ra J\tC.
$$
Since the dual of $g^*$ is the norm map  
$$
N_g = 1 +j: J\tC \ra JY,
$$
 this implies that
$$
\widehat {\theta_1}=\widehat {g^*} = N_{g} = 1+j: J\tC \ra JY \quad \mbox{and} \quad 
\widehat {\theta_2}=\widehat {g^*} \circ  \sigma^{-1} = \sigma^6 + \sigma j: J\tC \ra JY,
$$
since $\widehat \sigma = \sigma^{-1}$.

Now note that $\widehat \theta_1 \theta_1 = N_g  g^* = 2_{JY}$ and 
$\widehat \theta_2 \theta_2 = N_g \sigma^{-1} \sigma g^* = 2_{JY}$.
Hence the matrix of $\phi_{\psi^*\Xi}: JY \times JY \ra JY \times JY$ is 
$$
\phi_{\psi^*\Xi} = \left( \begin{array}{cc}
                             \widehat {\theta_1} \theta_1 &   \widehat {\theta_1} \theta_2\\
                             \widehat {\theta_2} \theta_1 &   \widehat {\theta_2} \theta_2
                                    \end{array} \right)
= \left( \begin{array}{cc}
                              2_{JY} & N_g \sigma g^* \\
                              N_g \sigma^{-1} g^* &  2_{JY}
                                    \end{array} \right).
$$
So we have,

\begin{prop} \label{polarization}
Let $g: \tC \ra Y = \tC/ \langle j \rangle$ be the natural map and
suppose $\widehat {J\tC} = J\tC$ and $\widehat {JY} = JY$ via the canonical principal polarizations. 
If we denote by $\varphi$ the isogeny $\varphi= N_g \sigma g^*: JY \ra JY$, then the polarization $\psi^*\Xi$ on $JY \times JY$ is given by the matrix
$$
\phi_{\psi^*\Xi} =  \left( \begin{array}{cc}
                              2_{JY} & \varphi \\
                              \widehat \varphi &  2_{JY}
                                    \end{array} \right).
$$
\end{prop}

\bigskip

In order to study the polarizations of type $(1,\dots,1,7)$
on the product $P:= JY \times JY$,  we recall from \cite[Section 5.2]{bl}  the description of the set of polarizations of degree 7 on $P$.  
According to \cite[Theorem 5.2.4]{bl} the canonical principal polarization on $JY \times JY$ induces a bijection between the 
sets of

(a) polarizations of degree 7 on $JY \times JY$ and

(b) totally positive symmetric endomorphism of $JY \times JY$ with analytic norm 7.\\

The endomorphisms of $JY \times JY$ are given by square matrices of degree 2 with entries 
in $\End(JY)$. The symmetry in (b) is with respect to the Rosati involution. Since the Rosati
involution with respect to the canonical polarization of $JY \times JY$ is just transposition of the 
matrices, we are looking for the symmetric positive definite matrices 
$$
A:= \left( \begin{array}{cc}
                  \rho_a(\alpha) & \rho_a(\beta)\\
                  \rho_a(\beta)^t & \rho_a(\delta)
          \end{array}   \right)
$$
with determinant 7, where $\alpha, \beta, \delta \in \End(JY)$.

\begin{prop}  \label{prop4.1}
If $P$ admits a polarization of degree 7, then $\End(JY) \supsetneqq \ZZ$.
\end{prop}

\begin{proof}
Suppose $\End(JY) = \ZZ$. For a general $Y$ the Rosati involution on $JY \times JY$ is 
transposition composed with the Rosati involution of the pieces $JY$, but  for a general $Y$
the Rosati involution on $JY$ is the identity.
Let $A$ be an endomorphism of degree 7, given by the matrix $A$. Then $\alpha, \beta$ and $\delta$ are integers and 
$\rho_a(\alpha) = \alpha_{\CC^3} = diag(\alpha, \alpha, \alpha)$ etc. and we have
$$
7 = \det A = \det \left( \begin{array}{cc}
                  diag(\alpha, \alpha, \alpha)  & diag(\beta,\beta,\beta)\\
                  diag(\beta,\beta, \beta) & diag(\delta,\delta, \delta)
          \end{array}   \right)
= (\alpha \delta - \beta^2)^3.
$$
Since this equation does not admit an integer solution, this give a contradiction.
\end{proof}

Recall that $\cR_{2,7}$ is an irreducible 3-dimensional variety and the Prym map is generically 
finite onto $\cB_D$. So also $\cB_D$ is of dimension 3. Since every element of $\cB_D$ is 
isomorphic to $JY \times JY$ for some curve $Y$ of genus 3, and since the number of decompositions $P = JY \times JY$ 
is at most countable, we get a 3-dimensional algebraic set, say $\cV$, of curves
$Y$ such that $JY \times JY$ admits a polarization of degree 7. In fact, as a consequence of the results in the paper 
$\cV$ is even a variety.

\begin{prop} \label{realmult}
The Jacobians $JY$ of all curves $Y \in \cV$ admit real multiplication in the totally real 
cubic number subfield $\QQ(\zeta_7)^0$ of the cyclotomic field $\QQ(\zeta_7)$.
\end{prop}

\begin{proof}
Since by Proposition \ref{prop4.1} $\End_{\QQ}(JY) \neq \QQ$, the Jacobian admits either
real, quaternion or complex multiplication (in the more general sense of \cite[Section 9.6]{bl}).
But an abelian variety with quaternion multiplication is even dimensional (\cite[9.4 and 9.5]{bl}).
If $JY$ admits complex multiplication by a skew field of degree $d^2$ over a totally complex
quadratic extension of a totally real number field of degree $e_0$, we would have 
(see \cite[9.6]{bl}) $3= d^2 e_0 m$ for some integer $m \geq 1$. So $d=1$ and
$e_0 = 3$ (by Proposition \ref{prop4.1}). Then $JY$ admits complex multiplication by a number 
field.
Since there are only countably many such abelian varieties, we conclude that $JY$ admits 
multiplication by a totally real number field $F$ of degree say $e$.

Then we have $3 = em$ for some positive integer $m$ \cite[\S 9.2]{bl} and Proposition 
\ref{prop4.1} implies $e=3$.  So $JY \times JY$ admits multiplication in $\mbox{SL}_2(F)$, with $F$ a totally real 
cubic number field. On the other hand $P$ admits a multiplication in the cyclotomic 
field $\QQ(\zeta_7)$ and this is a subfield of $\mbox{SL}_2(F)$ only if $F$ is the totally real cubic subfield $\QQ(\zeta_7)^0$ in 
$\QQ(\zeta_7)$. Therefore $F=\QQ(\zeta_7)^0$.

\end{proof}

As a consequence of the description of the polarization in  Proposition \ref{2.1} we have

\begin{prop}
Let $\cO$ denote the maximal order of the totally real cubic field $\End_{\QQ}(JY)$ and 
$\varphi: JY \ra JY$ the isogeny of Proposition \ref{polarization}. The following equation admits a 
solution in $\cO$
$$
7 = \det (4 \cdot \bf 1 - \rho_a( \varphi \widehat \varphi )).
$$

\end{prop}

\begin{thm} \label{unique}
For a general covering $f$ the decomposition $P(f) \simeq JY \times JY$ is unique up to automorphisms. 
\end{thm}

\begin{proof}
We may assume that the the ring of endomorphisms of $JY$ is the maximal order $\cO$ of the 
field $F$, since both families, the family of Prym varieties $P(f)$ and the family of Jacobians 
with multiplication in $F$ are irreducible of the same dimension 3. So 
$$
\End(P(f)) \simeq M_2(\cO),
$$
the ring of matrices of degree 2 with entries in $\cO$.  Let $A$ be any direct factor of $P(f)$.
So there is an abelian subvariety $B$ of $P(f)$ (necessarily isogenous to $A$) with
$$
P(f)  \simeq A \times B.
$$
We have to show that there is an automorphism $\alpha$ of $P(f)$ with $\alpha(JY) = A$. 

Recall that an element $\epsilon \in  \End(P(f))$ is idempotent if $\epsilon^2 = \epsilon$.
Now the direct factors of $P(f)$ correspond bijectively to the non-trivial idempotents of 
the endomorphism ring of $P(f)$, i.e. of $M_2(\cO)$. Namely, if $A$ is a direct factor, 
then the composition $\epsilon$ of the maps
$$
P(f) \simeq A \times B \stackrel{p_1}{\ra} A \stackrel{i_1}{\ra} A \times B \simeq P(f)
$$
is an idempotent of $\End(P(f))$. Here $p_1$ and $i_1$ are the natural projection and inclusion.
Conversely, if $\epsilon$ is an idempotent, the factor $A$ is given by the kernel of the 
endomorphism 
$1 -\epsilon$. Moreover, if $\epsilon' = \alpha \epsilon \alpha^{-1}$ with an automorphism 
$\alpha$ 
of $P(f)$, then 
$$
\alpha(1 -\epsilon) \alpha^{-1} = 1 - \epsilon'.
$$ 
Hence $\epsilon$ and $\epsilon'$ correspond to isomorphic direct factors. It suffices to show 
that all 
nontrivial idempotents (i.e. different from 0 and 1) are conjugate to each other, which is the 
content of the following lemma.
\end{proof}

\begin{lem}
Any $2$ nontrivial idempotents of $M_2(\cO)$ are conjugate to each other.
\end{lem}

\begin{proof}
Let $\epsilon \in M_2(\cO)$ be an idempotent. Its minimal polynomial $p(x)$ is different from
$x$ and $x-1$, since $\epsilon$ is nontrivial. 

Let $M := \cO^2$ with $\epsilon$ acting by multiplication in the natural way. Let $N_0$ be its
$0$-eigenspace, i.e. the kernel of $\epsilon$ and $N_1$ its 1-eigenspace, i.e. the kernel of 
$\epsilon -1$. We have
$$
\epsilon M \subset N_1 \quad \mbox{and} \quad (\epsilon -1) M \subset N_0,
$$
since $\epsilon^2 - \epsilon$ annihilates $M$. Moreover,
\begin{equation} \label{e7.1}
N_0 \cap N_1 = 0
\end{equation}  
\begin{equation} \label{e7.2}
N_0 + N_1 = M.
\end{equation}
The first equation is clear. For the second equation note that every $m \in M$ can be expressed
as $m = \epsilon m - (\epsilon -1)m$.

Neither $N_0$ nor $N_1$ can equal $M$, since otherwise we would have $p(x) = x$ or $x-1$.
Hence by \eqref{e7.2} neither can be zero, so $N_0$ and $N_1$ are both $\cO$-modules
of rank 1. Since $\cO$ is a principal ideal domain, there are $m_0, m_1 \subset M$ such that
$$
N_0 = \cO m_0 \quad \mbox{and} \quad N_1 = \cO m_1.
$$
It follows from \eqref{e7.1} and \eqref{e7.2} that $\{m_0, m_1\}$ is an $\cO$-basis of $M$.
With respect to this basis $\epsilon$ has the form 
$$
\left( \begin{array}{cc}
       0&0\\
       0&1
       \end{array}  \right).
$$
This implies the assertion.
\end{proof}

\begin{rem}
The proof works more generally for $M_2(\cO)$ with $\cO$ any principal ideal domain.
\end{rem}

\section{The number of isomorphism classes of coverings $\overline f$}

Let $f: \tC \ra C$ be an \'etale cyclic covering of degree 7 of a smooth curve of genus 2.  As we saw in Section 2,
the lifting $j$ of the hyperelliptic involution $\iota$ of $C$ is not unique and in fact 
there are exactly 7 liftings $j_0, \dots,j_6$ which are the 7 involutions of $D_7$.
If $g_i: \tC \ra Y_i := \tC/\langle j_i \rangle$ denotes the 
quotient map given by the involution $j_i:=  j\sigma^i$ (so $g_0=g$) we have the following cartesian diagram.
\begin{equation} \label{eq10.1}
\xymatrix{
& \tC \ar[dl]_f \ar[dr]^{g_i} & \\
C \ar[dr]_h  & &Y_i \ar[dl]^{\overline f_i}  \\
&  \PP^1 &
}
\end{equation}
In the previous section we saw that $Y_i$ is of genus 3 and $g_i$ 
is ramified exactly at one point over each Weierstrass point of $C$. Since $f$ is \'etale, the 
commutativity of \eqref{eq10.1} implies that the branch points of $\overline f_i:Y_i \ra \PP^1$
coincide with the branch points of $h$ and each branch point is of ramification type 
$(2,2,2,1)$, i.e $\overline f_i$ is simply ramified at 3 points and unramified at 1 point over the 
branch point. So the branch points of $\overline f_i$ coincide for all $i$. 
If $p_1, \dots, p_6 \in \PP^1$
are the branch points, then the permutations corresponding to the fibers 
$(\overline f_i)^{-1}(p_j)$
are pairwise conjugate within $D_7$ for all $i=0,\dots,6$ and fixed $j$ with the same conjugation 
for  
$j = 1, \dots,6$ and such that the non-ramified points of $f_j$ over each $p_i$ are pairwise 
different 
for $j = 0, \dots,6$.

Note that the coverings $\overline f_i$ correspond to conjugate subgroups of $D_7$.
Hence the $\overline f_i: Y_i \ra \PP^1$ are isomorphic coverings.

\begin{lem} \label{lem10.1}
For any general smooth curve $C$ of genus $2$ with double covering $h: C \ra \PP^1$, there is a 
canonical bijection between the sets of
\begin{enumerate}
\item[(a)] isomorphism classes of \'etale cyclic coverings $f:\tC \ra C$ of degree $7$,
\item[(b)] isomorphism classes of degree-$7$ coverings $\overline f:Y \ra \PP^1$ of 
ramification type $(2,2,2,1)$  over each of the 6 ramification points.
\end{enumerate}
\end{lem}

\begin{proof}
We saw above that any covering $f$ in (a) gives 7 coverings in (b), which are all isomorphic to 
each other. Conversely, let 
$\overline f:Y \ra \PP^1$ be one of  the coverings in (b). Since $C$ 
is general, the monodromy group $G$ of $\overline f$ coincides with the Galois group of the Galois closure of $Y \ra \PP^1$.
Define
$$
\tC := C \times_{\PP^1} Y.
$$
According to the Lemma of Abhyankar \cite[Lemma 2.14]{p} the projection $\tC \ra C$ is 
\'etale. So $\tC$ is smooth and it is easy to see that it is Galois over $\PP^1$ with Galois group 
$\simeq D_7$. 
Hence $\tC$ is the Galois closure of $\overline f$. In particular $\tC \ra C$ is cyclic
\'etale of degree 7. Clearly  both constructions are inverse to each other and isomorphic coverings in (a) correspond to isomorphic coverings in (b). 
\end{proof}

Given a covering $f: \tC \ra C$ corresponding to a general element in $\cR_{2,7}$, by Lemma \ref{lem10.1}
there is a corresponding degree 7 map $Y \ra \PP^1$ of ramification type $(2,2,2,1)$ and, according to
Proposition \ref{2.1}, the Prym variety $P(f)$ decomposes as 
$$
P(f) = JY \times JY.
$$

Now there are $7^4 -1 = 2400$ non-trivial 7-torsion points of $C$. 
Since a cyclic \'etale cover $\tC \ra C$ of degree 7 is given by a cyclic subgroup of order 7 of $JC[7]$,
there are exactly 
$$
\frac{2400}{6} = 400
$$
isomorphism classes of cyclic \'etale coverings $\tC \ra C$. 

Let $p_1, \dots,p_6$ be the branch points of the hyperelliptic covering $C \stackrel{h}{\lra} \PP^1$.
Let $N$ be the number of isomorphism classes of coverings $\overline f: Y \ra \PP^1$ ramified exactly over
$p_1, \dots, p_6$ of ramification type $(2,2,2,1)$. If we denote by $\beta: \cR_{2,7} \ra \cM_2$ the forgetful map
onto the moduli space of smooth curves of genus 2 
and the intersection $\pr^{-1}(P(f)) \cap \beta^{-1}([C]) $ consists of $d$ elements, then Lemma \ref{lem10.1} implies that
\begin{equation} \label{number-cov}
400 = d \cdot N.
\end{equation}


Let $p_1, \dots, p_6$ denote 6 points of $\PP^1$ in general position.  We want to compute the number $N$ of isomorphism classes of degree 7 
coverings ramified of ramification type 
$(2,2,2,1)$ exactly over the points $p_1,\dots, p_6$.

Let $\pi_1$ denote the monodromy group of $\overline f$ with base point 
$p \neq p_i$ for all $i=1,\dots,6$. We consider $\pi_1$ as a subgroup of the symmetric group 
$\cS_7$ of degree 7 acting on the set $\{1,\dots,7\}$. We need the following 2 trivial lemmas
(which  can of course be formulated for any odd prime $p$ instead of $7$).

\begin{lem}  \label{l2.1}
Let $a$ and $b$ be involution of $\cS_7$ such that 
$$
s:= ab
$$
is a cycle of order $7$, then
the group generated by $a$ and $b$ is isomorphic to $D_7$. Conversely, any subgroup of $\cS_7$
isomorphic to $D_7$ is of this type.
\end{lem}

\begin{proof}
For the first assertion we have to show that $asa=s^{-1}$. But
$$
as = aab = b = b^{-1} = b^{-1} a^{-1}a = s^{-1} a.
$$
The second assertion is obvious.
\end{proof}

\begin{lem} \label{l2.2}
Let $a$ and $b$ be involutions of $\cS_7$ such that $\langle a,b \rangle \simeq D_7$. Then $a$ and 
$b$ are products of $3$ disjoint transpositions such
that no transposition in $a$ occurs in $b$.
Equivalently $a$, respectively $b$, fixes exactly one number $i$,  respectively $j$, and $i \neq j$.
 Each of the numbers $1,\dots,7$ is fixed by exactly one of the $7$ involutions.
\end{lem}

\begin{proof}
It is easy to check that if $a$ or $b$ consist of less that 3 disjoint transpositions or if one 
transposition occurring in $a$ occurs also in $b$, then $s = ab$ is not of order 7, which contradicts 
Lemma \ref{l2.1}. For the last assertion note that if 2 involutions would fix the same number $i$,
then all involutions and hence all elements of the group would fix $i$.
\end{proof}

We give an example, although this is not necessary for the sequel, but makes the following 
proposition perhaps clearer.

\begin{exam}
Consider the following involutions of $\cS_7$ (with right action, i.e. the product is from left to right)
$$
a = (12)(34)(56) \quad \mbox{and} \quad b= (23)(45)(67),
$$
with product
$$
s:= ab = (1357642).
$$
According to Lemma \ref{l2.1} the group $\langle a,b \rangle$ is isomorphic to $D_7$.
\end{exam}

\begin{prop} \label{p2.4}
All dihedral subgroups $D_7$ of $\cS_7$ are conjugate to each other.
\end{prop}
 
\begin{proof}
Different labelings of the set on which $\cS_7$ acts give conjugate subgroups of $\cS_7$. We label this set in such a way that 
$$
a = (12)(34)(56).
$$
The labelling is unique up to the transpositions $(12), (34)$ and $(56)$. Moreover,
according to Lemma \ref{l2.2} we may choose $b$ such that it fixes 1. The involution $b$ cannot
contain the transposition $(27)$, since otherwise $s = ab$ would contain the cycle $(172)$.
So we remain with the following possibilities $b_i$ for $b$:
$$
b_1= (23)(45)(67), \quad b_2 = (23)(46)(57), \quad b_3=(23)(47)(56),
$$
$$
b_4= (24)(35)(67), \quad b_5 = (24)(36)(57),  \quad b_6=(24)(37)(56),
$$
$$
b_7 = (25)(34)(67), \quad b_8 = (25)(36)(47), \quad b_9= (25)(37)(46),
$$
$$
b_{10} = (26)(34)(57), \quad b_{11} = (26)(35)(47), \quad b_{12} = (26)(37)(45).
$$
We compute
$$
s_1= ab_1= (1357642), \quad s_2= ab_2 = (1367542), \quad s_3 = a b_3 = (13742),
$$
$$
s_4= ab_4= (1457632), \quad s_5= ab_5 = (1467532), \quad s_6 = a b_3 = (14732),
$$
$$
s_7= ab_7= (15762), \quad s_8= ab_8 = (1537462), \quad s_9 = a b_9 = (1547362),
$$
$$
s_{10}= ab_{10}= (16752), \quad s_{11}= ab_{11} = (1637452), \quad s_{12} = a b_{12} = (1647352).
$$
So exactly the groups $G_i:= \langle a,b_i \rangle$ with $i = 1,2,4,5,8,9,11$ and 12 are isomorphic to $D_7$.

We may still conjugate the subgroups with the transpositions $(12), (34)$ and $(56)$. One checks that
$$
(12)s_1(12) = s_5^6 \in G_6, \quad (34)s_1(34) = s_4 \in G_4, \quad (56)s_1(56) = s_2 \in G_2,
$$
$$
(12)s_8(12) = s_{12}^6 \in G_{12}, \quad (34)s_8(34) = s_9 \in G_9, \quad (56)s_8(56) = s_{11} \in G_{11}.
$$
This implies that $G_1,G_2,G_4$ and $G_5$ as well as $G_8,G_9,G_{11}$ and $G_{12}$
are pairwise conjugate. Hence it suffices to show that $G_1$ is conjugate to $G_8$. But one easily checks that with the permutation $p:= (2463)$ we have
$$
p a p^{-1} = a \quad \mbox{and} \quad p b_1 p^{-1} = b_8.
$$
which gives the assertion.
\end{proof}

\begin{prop} \label{prop3.5}
Given $6$ points of $\PP^1$ in general position. There are exactly $400$ isomorphism classes of degree-7 coverings 
$\overline f: Y \ra \PP^1$ with monodromy group $D_7$  ramified of type $(2,2,2,1)$
over each of the $6$ points. 
\end{prop} 

\begin{proof}
Let $p_1, \dots,p_6\in \PP^1$ be 6 points in general position.
According to Proposition \ref{p2.4} all subgroups isomorphic to $D_7$ are conjugate. 
For example, we may take the monodromy subgroup of the covering to be
$$
\pi_1 = \langle a, b_1 \rangle = \{ 1, s_1, \dots, s_1^6, a, as_1, \dots, as_1^6 \}.
$$
One has to compute the number of 
conjugacy classes of 6 involutions $c_i$ of type (2,2,2,1) (associating $c_i$ to the point $p_i$)
such that 
\begin{equation} \label{e3.1}
c_1 c_2 c_3 c_4 c_5 c_6 = 1.
\end{equation}
The elements of  $D_7$ are either an involution (an odd permutation) or an element of order 7 (an even permutation). 
Hence the product of 5 involutions in $D_7$  gives an odd permutation, so it is an involution. Thus, for any involutions 
$c_1, \dots, c_5$ of $\pi_1$, $c_6:= c_1\cdots c_5$ is an involution satisfying \eqref{e3.1}. Since any 2 involutions of $\pi_1$
generate the group, only 7 of the corresponding coverings are not connected, namely those with 
$c_1 = \cdots = c_6$. This gives $7^5 -7$ coverings. Two $6$-tuples of involutions give 
isomorphic coverings if and only if they are conjugate. 

In order to compute the set of conjugate classes of 6-tuples of involutions, recall that the only 
transitive subgroups of $\cS_7$ which are not contained in $\cA_7$ are subgroups isomorphic to
\begin{itemize}
\item the dihedral group $D_7$,
\item the group $L_7 = AGL_1(\FF_7)$ of affine transformations of the line with 7 points,
\item $\cS_7$ itself.
\end{itemize}
(For the convenience of the reader we give a proof of this statement: 
Let $G$ be a proper subgroup
of $\cS_7$ of this type. Clearly $G$ is a soluble group, since it is not 
contained in $\cA_7$. It is primitive, being a
permutation group of degree 7. Hence, according to \cite[7.2.7]{r}, it 
is a subgroup of $L_7$.
Since $D_7$ is the only transitive subgroup of $L_7$, this gives the 
statement.)
The group $D_7$ is not normal in $\cS_7$, but it is normal in $L_7$. Since on the other hand 
$L_7$ is isomorphic to the group of outer automorphisms of $D_7$, two 
6-tuples of involutions give isomorphic coverings if and only if they are conjugate under an
element of $L_7$. Since $L_7$ is of order 42, we finally get
$$
\frac{7^5 -7}{42} = 400
$$
isomorphism classes of degree 7 coverings $Y \ra \PP^1$ of ramification type $(2,2,2,1)$.
\end{proof}

Together with \eqref{number-cov} we get

\begin{cor} \label{cor3.8}
Let $\overline f: Y \ra \PP^1$ be a degree 7 covering with monodromy group $D_7$  ramified of type $(2,2,2,1)$ over $6$ general points. 
There is exactly one cyclic \'etale cover 
$f: \tC \ra C$ such that the diagram \eqref{eq10.1} is commutative. In particular, 
$P(f) \simeq JY \times JY$ as polarized abelian varieties of type $(1,\dots,1,7)$.
\end{cor}

\section{Proof of the main theorem}

Let $\cC_3$ denote the locus in $\cM_3$ of genus 3 curves admitting a degree 7 covering 
$\overline f: Y \ra \PP^1$ with ramification type $(2,2,2,1)$ over 6 general points in $\PP^1$. It is an irreducible variety of dimension 
3 since it is an image of the moduli space of pairs $(Y, \overline f)$, which according to 
Lemma \ref{lem10.1} are in bijection with the elements in $\cR_{2,7}$. 

\begin{thm} 
The elements of a generic fiber of the Prym map $\pr:\cR_{2,7} \ra \cB_D$ are in bijection with the set of degree 7 coverings 
$\overline f: Y \ra \PP^1$ with ramification type $(2,2,2,1)$ over 6 general points that a curve $Y \in \cC_3$ admits. 
\end{thm}

\begin{proof}
By Corollary \ref{cor3.8} the elements in the fiber of $P(f) \in \cB_D$ are necessarily in different fibers of the forgetful map $\beta:  \cR_{2,7} \ra \cM_2$.
Together with uniqueness of the decomposition $P(f) \simeq JY \times JY  $ as polarized abelian varieties
(see Theorem \ref{unique}) this implies that the covering $f:\tC \ra C$ is completely determined by the pair 
$(Y, \overline{f})$, since the branch points of $\overline{f}$ determine the genus 2 curve $C$.
So the elements in $\cR_{2,7}$ with Prym variety isomorphic to $JY \times JY$ are in bijection with the coverings 
$\overline f: Y \ra \PP^1$ with ramification type $(2,2,2,1)$ for a fixed curve $Y \in \cC_3$.
\end{proof}

It has been shown in \cite[Theorem 1.1]{lo2} that the degree of the Prym map $\pr$ is 10, so we get as immediate corollary

\begin{cor}
There are at most 10 coverings $\overline f: Y \ra \PP^1$ of degree 7 with ramification type $(2,2,2,1)$ for a curve $Y \in \cC_3$.
\end{cor}

\end{document}